\documentclass[9pt]{article} 
\usepackage{hyperref}
\usepackage{url}
\usepackage{smile}
\usepackage{graphicx} 
\usepackage{algorithm}
\usepackage{algorithmic}
\usepackage{epstopdf}
\usepackage[margin=1.0in]{geometry}
\usepackage[export]{adjustbox}
\usepackage{mathtools, cuted}
\usepackage{bbm}
\usepackage{wrapfig}
\usepackage{subcaption}
\usepackage{caption}
\usepackage{enumitem}
\usepackage{tabularx}
\usepackage{smile}
\usepackage{tcolorbox}
\usepackage{enumitem}
\usepackage{mathtools, cuted}
\usepackage{caption}
\usepackage{subcaption}

\numberwithin{equation}{section}

\usepackage[nocompress]{cite}

\hypersetup{
    colorlinks=true,
    linkcolor=blue,
    anchorcolor=blue, 
    citecolor=blue,
}

\allowdisplaybreaks

\newcommand{\revise}[1]{{\color{black}#1}}

\renewcommand{\frac}[2]{\tfrac{#1}{#2}}

\title{
 Computational Hardness of Static Distributionally Robust Markov Decision Processes
    }
    
\author{
    Yan Li   \thanks{Department of Industrial and Systems Engineering,
Texas A\&M University. (E-mail: \url{yan.li@tamu.edu}).}
}
\date{\vspace{-3ex}}

\begin{document}
{
\makeatletter
\addtocounter{footnote}{1} 
\renewcommand\thefootnote{\@fnsymbol\c@footnote}%
\makeatother
\maketitle
}

\maketitle

\begin{abstract}
We present some hardness results on finding the optimal policy for the static formulation of distributionally robust Markov decision processes.
We construct problem instances such that when the considered policy class is Markovian and non-randomized, finding the optimal policy is NP-hard.  When the considered policy class is Markovian and randomized, the robust value function possesses sub-optimal strict local minimizers, and finding the optimal policy is also NP-hard. The considered instances involve an ambiguity set with only two transition kernels.  
\end{abstract}


\section{Introduction}\label{sec_intro}

Consider a finite-horizon distributionally robust Markov decision process  $(\cbr{\cS_t}_{t=1}^T, \cbr{\cA_t}_{t=1}^T, \cP, \cbr{c_t}_{t=1}^T)$.
Here $\cS_t$ and $\cA_t$ denote the finite state and action spaces at stage $t$, 
$\cP$ denotes a compact set of well-defined transition kernels, 
and $c_t: \cS_t \times \cA_t \to \RR \cup \cbr{+\infty}$\footnote{We allow the cost to take the value of infinity to indicate potential dependence of the action space on the state.} denotes the cost function at stage $t$. 
For any policy $\pi$, its corresponding value function associated with a kernel $P \in \cP$ is given by 
\begin{align}\label{def_nonrobust_value}
V^\pi_P(s_1) = \EE^{\pi, P} \sbr{ \tsum_{t=1}^T c_t(S_t, A_t) | S_1 = s_1},
\end{align}
where $\EE^{\pi, P}\sbr{\cdot}$ denotes the expectation with respect to the law of stochastic process $\cbr{(S_t, A_t)}_{t=1}^T$ induced by $(\pi, P)$ (cf.  Ionescu Tulcea Theorem \cite{tulcea1949mesures}). 
The static formulation of distributionally robust Markov decision process (DRMDP) \cite{wiesemann2013robust,Tallec,Nilim2005,iyen} defines the robust value function as
\begin{align}\label{def_robust_value}
V^\pi_r(s_1) = \max_{P \in \cP} V^\pi_P(s_1),
\end{align}
for a given initial state $s_1 \in \cS_1$.
Correspondingly the policy optimization problem considers   
\begin{align}\label{def_static_rmdp}
\min_{\pi \in \Pi} V^\pi_r(s_1) = \min_{\pi \in \Pi} \max_{P \in \cP} V^\pi_P(s_1),
\end{align}
where $\Pi$ is a given set of Markovian policies.\footnote{
It is worth mentioning here that optimal policies for the static formulation \eqref{def_static_rmdp} can be history-dependent. 
We restrict our attention to Markovian policies here as history-dependent policies have a space complexity that can grow exponentially with respect to the horizon length. 
}
\revise{
Note that since $V^{\pi}_P(s)$ is continuous jointly in $(\pi, P)$, the corresponding minimization and maximization problems in \eqref{def_static_rmdp} are attainable. 
}

\revise{
We refer to \eqref{def_static_rmdp} as the static formulation of DRMDPs, since the kernel defining the expectation in \eqref{def_static_rmdp} is selected {\it prior} to the realization of the process $\cbr{(S_t, A_t)}_{t=1}^T$. 
In the current literature this is often referred to as  {\it robust MDPs}.
We emphasize the term {\it static} here as this is to be contrasted with the dynamic formulation introduced in \cite{li2025rectangularity}, which considers an adversarial nature that {\it dynamically} selects the kernels based on the process history.
It should be noted that while dynamic programming equations for the static formulation require various notions of rectangularity on $\cP$ \cite{wiesemann2013robust,Tallec,Nilim2005,iyen}, 
 they can be established for the dynamic formulation with no such requirements.
In addition, \cite{li2025rectangularity} shows that whenever dynamic programming for the static formulation exists, the static formulation reduces to the dynamic formulation, and using this observation one can reformulate \eqref{def_static_rmdp} so that the corresponding ambiguity set becomes $\mathrm{s}$-rectangular \cite{le2007robust, wiesemann2013robust}, while retaining the same optimal value function and optimal policy. 
Here, $\cP$ is said to be $\mathrm{s}$-rectangular if the selection of worst-case transition probabilities in \eqref{def_robust_value} can be made independently for each state.
A separate notion of distributionally robust MDP framework in \cite{xu2010distributionally} considers a nature choosing a distribution over the set of kernels, and to derive dynamic programming equations, the rectangular structure on the ambiguity set is assumed. It can be seen from \cite[Remark 2.5]{li2025rectangularity} that this formulation can be recovered by the dynamic formulation when the nature adopts a randomized strategy over the ambiguity set \cite{li2025rectangularity}.
}

In this manuscript we will focus on a subclass of \eqref{def_static_rmdp}, where $\cP = \cbr{P_{(1)}, P_{(2)}}$ consists of only two distinct transition kernels. 
Clearly,  \eqref{def_static_rmdp} reduces~to 
\begin{align}\label{static_rmdp_two_kernels}
\min_{\pi \in \Pi} \max \cbr{V^\pi_{P_{(1)}}(s_1), V^\pi_{P_{(2)}}(s_1) }.
\end{align}
Note that $\cP$ is non-rectangular \cite{le2007robust,wiesemann2013robust} unless $P_{(1)}$ and $P_{(2)}$ differ at only one state. 

\revise{
It has been known that for robust MDP \eqref{def_static_rmdp} with non-rectangular ambiguity sets, even evaluating $V^\pi_r$ becomes NP-hard in the infinite-horizon setting \cite{wiesemann2013robust}. 
For the finite-horizon setting, we believe a similar hardness result can be established by reducing from the bilinear minimization problem.}
Nevertheless, it is worth mentioning here that despite $\cP$ being non-rectangular in \eqref{static_rmdp_two_kernels}, the corresponding robust policy evaluation problem \eqref{def_robust_value}  involves solving two linear systems and hence can be solved in polynomial time. 
\revise{In view of this, our main problem of interest in this manuscript is to focus on potential sources of hardness that solely come from the policy optimization. 
While hardness results of optimizing robust MDPs have been established when the ambiguity set contains $K$ kernels (\cite{le2007robust, steimle2021multi}, see also Remark \ref{remark_multi}),  
such results are based on reduction from multi-model MDPs, whose main source of difficulty comes from its continuous belief state that has dimension $K$ and the ensuing curse of dimensionality.
It remains unclear to us whether the existing hardness results can be applied to \eqref{static_rmdp_two_kernels}.
}

\revise{
The main finding of this manuscript is to show that despite polynomial solvability of robust policy evaluation \eqref{def_robust_value}, the robust policy optimization itself for \eqref{static_rmdp_two_kernels} remains challenging. 
In particular, we show that finding the optimal non-randomized Markovian policy is NP-hard, based on a simple reduction from the set partition problem. 
On the other hand, when searching within the randomized Markovian policy class, we show that the robust objective \eqref{static_rmdp_two_kernels} possesses strict local minimizers, and finding the optimal policy is also NP-hard. 
A direct implication of this is, perhaps surprisingly, that the gradient-based approaches (e.g., \cite{agarwal2021theory, sutton1999policy}) enjoying global optimality guarantees for non-robust MDPs  do not converge for the simple robust MDP instance \eqref{static_rmdp_two_kernels} above.  
We also discuss briefly settings and conditions for which global optimality guarantees can be retained, and the difficulty in certifying some of these conditions. 
}

\revise{{\bf Notations.}}
Let us use $\Pi_{\mathrm{MD}}$ to denote the set of non-randomized Markovian  policies,
and use $\Pi_{\mathrm{MR}}$ to denote the set of randomized Markovian  policies.
To facilitate our discussion, for any integer $n > 0$, let us denote $[n] = \cbr{0, \ldots, n-1}$. 
We reserve $\norm{\cdot}$ for the standard Euclidean norm. 
\revise{For a finite set $X$, we use $\Delta_{X}$ to denote the probability simplex over $X$.}


\section{Computational Hardness}\label{sec_nonrandomized}

We begin by showing that finding the optimal non-randomized Markovian policy is NP-hard.

\begin{theorem}[NP-hardness for $\Pi_{\mathrm{MD}}$]\label{thrm_np_nonrandom}
The robust MDP  \eqref{static_rmdp_two_kernels} is NP-hard
when $\Pi = \Pi_{\mathrm{MD}}$. 
\end{theorem}

\begin{proof}
\revise{Our argument will be based on reduction from the set partition problem (NP-complete), which asks, for any given set $\cW = \cbr{w_1, w_2, \ldots, w_n}\subset \ZZ_+$, whether there exists a subset $\cW' \subset \cW$ such that 
$\tsum_{w \in \cW'} w = \tsum_{w \in {\cW \setminus \cW'}} w$.}
Now let us consider the following construction of a robust MDP instance.
For any $1 \leq t \leq n$, define  
$\cS_{2t-1} = \cbr{2t-1}$,
$\cS_{2t} = \cbr{0, 1}$, 
 $\cA_{2t-1} = \cbr{0,1}$,
 $\cA_{2t} = \emptyset$.
Let the cost function be defined as 
\begin{align*}
 c_{2t-1} (\cdot) = 0,
c_{2t} (0) = w_t,  ~ c_{2t}(1) = 0.
\end{align*} 
Consider two transition kernels defined by 
\begin{align*}
& P_{(1)}(S_{2t} = 0 | S_{2t-1} = 2t-1 , A_{2t-1} = 0) = 1 , ~ P_{(1)}(S_{2t} = 1 | S_{2t-1} = 2t-1 , A_{2t-1} = 1) = 1 ,  \\
& P_{(2)}(S_{2t} = 1 | S_{2t-1} = 2t-1 , A_{2t-1} = 0) = 1 , ~ P_{(2)}(S_{2t} = 0 | S_{2t-1} = 2t-1 , A_{2t-1} = 1) = 1 ,
\end{align*}
and $ P_{(1)}(S_{2t+1} = 2t+1 | S_{2t}  )  =  P_{(2)}(S_{2t+1} = 2t+1 | S_{2t}  ) = 1$. 
Given any non-randomized policy $\pi \in \Pi_{\mathrm{MD}}$, let 
$\cS_\pi = \cbr{t: \pi(S_{2t-1}) = 0, 1 \leq t \leq n}$. We have 
\begin{align*}
V^\pi_{P_{(1)}}(s_1) = \tsum_{t \in \cS_\pi} w_t ,
~ 
V^\pi_{P_{(2)}}(s_1) = \tsum_{t \notin \cS_\pi} w_t  ,
\end{align*}
which implies 
\begin{align*}
V^\pi_r(s_1) = \max \cbr{
\tsum_{t \in \cS_\pi} w_t, \tsum_{t \notin \cS_\pi} w_t
}.
\end{align*}
Hence we obtain  
$\min_{\pi \in \Pi_{\mathrm{MD}}} V^\pi_r(s_1) \geq \frac{\tsum_{ t = 1}^n w_t}{2}$, 
and equality holds if and only if there exists 
$\cW' \subset \cW$ such that 
$\tsum_{w \in \cW'} w = \tsum_{w \in {\cW \setminus \cW'}} w$. 
\end{proof}



We then construct problem instances for which \eqref{static_rmdp_two_kernels} possesses a sub-optimal local minimizer among the set of randomized Markovian policies.

\revise{
\begin{definition}
We say a policy $\pi \in  \Pi_{\mathrm{MR}}$ is a sub-optimal strict local minimizer of \eqref{static_rmdp_two_kernels} if there exists $\delta > 0$ and $\pi^* \in \Pi_{\mathrm{MR}}$ such that 
\begin{align*}
V^{\pi^*}_r(s_1) <  V^\pi_r(s_1) < V^{\pi'}_r(s_1), ~ \forall \pi' \neq \pi, \norm{\pi' - \pi} \leq \delta,  \pi' \in \Pi_{\mathrm{MR}},
\end{align*}
where $\norm{\pi' - \pi} = \sup_{t} \sup_{s_t \in \cS_t}  \norm{\pi_t(\cdot|s_t) - \pi_t'(\cdot|s_t)}$.
\end{definition}
}

\begin{theorem}[Local minimizer for $\Pi_{\mathrm{MR}}$]\label{thrm_strict_local_min}
There exists a  robust MDP instance such that \eqref{static_rmdp_two_kernels} has a sub-optimal strict local minimizer 
when $\Pi = \Pi_{\mathrm{MR}}$. 
\end{theorem}

\begin{proof}
Given an integer $n > 0$ and $A \in \RR^{n \times n}$, 
consider a three-stage Markov decision process with $\cS_1 = \cA_1 = \cS_2 = \cA_2 = [n]$, $\cS_3 = [n] \times [n]$, $\cA_3 = \emptyset$. 
The per-stage cost function is given by $c_1(\cdot) = c_2(\cdot) = 0$, and $c_3(i, j ) = A_{ij}$ for $(i,j) \in [n] \times [n]$. 
Consider transition kernel $P_{(1)}$  defined as 
\begin{align*}
P_{(1)} (S_2 = i| S_1 = s_1, A_1 = i) = 1, ~ 
P_{(1)}( S_3 = (i,j) | S_2 = i, A_2 = j) = 1 , ~ \forall i, j \in [n].
\end{align*}
That is, starting from $s_1$, with action $A_1 = i$, the state transits to $S_2 = i$. Upon taking action $A_2 = j$, the state transits to $S_3 = (i,j)$. 
In addition, the transition kernel $P_{(2)}$ is defined as  
\begin{align*}
P_{(2)} (S_2 = n-1-i | S_1 = s_1, A_1 = i) = 1, ~ 
P_{(2)} ( S_3 = (n-1-i,j) | S_2 = i, A_2 = j) = 1 , ~ \forall i, j \in [n].
\end{align*}
That is, starting from $s_1$, with action $A_1 = i$, the state transits to $S_2 = n-1-i$. Upon taking action $A_2 = j$, the state transits to $S_3 = (i,j)$.

For notational simplicity, given a randomized policy $\pi = (\pi^1, \pi^2)$ with $\pi^t: \cS_t \to \Delta_{\cA_t}$, let us denote 
\begin{align*}
\pi^1_i = \pi^1(A_1 = i | S_1 = s_1), 
\pi^2_{ij} = \pi^2 (A_2 = j | S_2 = i), ~ \forall i,j \in [n].
\end{align*}
The value function $V^\pi_{(1)}$ associated with kernel $P_{(1)}$ is given by 
\begin{align*}
V^\pi_{(1)}(s_1) = 
\tsum_{i \in [n]} \pi^1_i \rbr{ \tsum_{j \in [n]} \pi^2_{ij} A_{ij} } ,
\end{align*}
and the value function $V^\pi_{(2)}$ associated with kernel $P_{(2)}$ is 
\begin{align*}
V^\pi_{(2)}(s_1) = 
\tsum_{i \in [n]} \pi^1_i \rbr{ \tsum_{j \in [n]} \pi^2_{(n-1-i) j} A_{ij} } .
\end{align*}
Hence the robust value function is defined as 
\begin{align*}
V^\pi(s_1) = \max \cbr{
\tsum_{i \in [n]} \pi_i^1 \rbr{ \tsum_{j \in [n]} \pi_{ij}^2 A_{ij} }, 
\tsum_{i \in [n]} \pi_i^1 \rbr{ \tsum_{j \in [n]} \pi_{(n-1-i) j}^2 A_{ij} }
}.
\end{align*}
Now consider $n = 2$ 
and 
\begin{align*}
A = \begin{bmatrix}
1 & 0 \\
-1 & 1
\end{bmatrix}.
\end{align*}
Consequently   we have
\begin{align}\label{def_robust_val_n_1}
V^\pi(s_1) = \max \cbr{
\pi^1_0 \pi^2_{00} + \pi^1_1 \rbr{ -\pi^2_{10} + \pi^2_{11} } ,
~
\pi^1_0 \pi^2_{10} + \pi^1_1 \rbr{ -\pi^2_{00} + \pi^2_{01} } 
}.
\end{align}
Let us define  
 \begin{align}\label{partial_min_pi_2}
 f(\pi^1) = \min_{\pi^2 \revise{\in \Delta_{\cA_2}} } V^\pi(s_1) 
 = \min_{\pi^2_{00}, \pi^2_{10} \in [0,1]} g_{\pi^1}(\pi^2_{00}, \pi^2_{10}) ,
 \end{align}
where 
$g_{\pi^1} (a,  b) 
= \max \cbr{
 \pi^1_0 a + \pi^1_1 \rbr{ 1 - 2 b } ,
\pi^1_0 b + \pi^1_1 \rbr{ 1 -2 a } 
 }
 $. 
 Note that $g_{\pi^1}(\cdot)$ is convex and symmetric, i.e., $g_{\pi^1}(a,b ) = g_{\pi^1} (b, a)$. 
 Hence there must exist an optimal solution $(a^*, b^*)$  of \eqref{partial_min_pi_2} satisfying $a^* = b^*$,~and 
 \begin{align}\label{partial_min_wrt_pi_1}
  f(\pi^1) = \min_{\pi^2 \revise{ \in \Delta_{\cA_2}}} V^\pi(s_1) 
  = \min_{a \in [0,1]} \pi_0^1 a + \pi_1^1(1-2a) 
    = 
  \min \cbr{\pi^1_1, \pi^1_0 - \pi^1_1} .
 \end{align}
 From the above observation, it is then clear that $\overline{\pi}^1 = (1, 0)$ is a strict local minimizer of $f(\cdot)$.
 That is, there exists $\delta > 0$, such that 
 $f(\pi^1) > f(\overline{\pi}^1)$ for any $\norm{\pi^1 - \overline{\pi}^1} \leq \delta$ and $\pi^1 \neq \overline{\pi}^1$. 
 In addition, $\overline{\pi}^1$ is sub-optimal for $f(\cdot)$ as
 \begin{align*}
 \min_{\pi \revise{\in \Pi_{\mathrm{MR}}}} V^\pi(s_1)  = \min_{\pi^1 \revise{\in \Delta_{\cA_1} }} f(\pi^1) = -1 < 0 = f(\overline{\pi}^1).
 \end{align*}
 Let $\overline{\pi}^2 \in \Argmin_{\pi^2 \revise{\in \Delta_{\cA_2}}} V^{\overline{\pi}^2, \pi^2}(s_1)$. 
  That is, 
 $f(\overline{\pi}^1) = V^{\overline{\pi}^1, \overline{\pi}^2}(s_1)$. 
 Note that from \eqref{partial_min_pi_2}, the definition of $g_{\overline{\pi}^1}(\cdot)$ and the choice of $\overline{\pi}^1$,  we have that $\overline{\pi}^2$ is unique. 
 Let us write $\overline{\pi} \coloneqq (\overline{\pi}^1, \overline{\pi}^2)$. 

We proceed to show that $(\overline{\pi}^1, \overline{\pi}^2)$ is a strict local minimizer of $V^\pi(s_1)$ but not a global minimizer. 
Indeed, consider any $\pi \coloneqq (\pi^1, \pi^2) $ such that $\norm{\pi - \overline{\pi} } \leq \delta$ and $\pi \neq \overline{\pi}$.
If $\pi^1 \neq \overline{\pi}^1$, then  
\begin{align*}
V^{\pi^1, \pi^2} (s_1) 
\overset{(a)}{\geq}  f(\pi^1) 
\overset{(b)}{>}  f(\overline{\pi}^1) 
\overset{(c)}{=} V^{\overline{\pi}^1, \overline{\pi}^2} (s_1) 
\overset{(d)}{>}  \min_{\pi \revise{\in \mathrm{\Pi}_{\mathrm{MR}}}} V^\pi(s_1) ,
\end{align*}
where $(a)$ follows from the definition of $f(\cdot)$ in \eqref{partial_min_pi_2}, 
 $(b)$ follows from that $\overline{\pi}$ is a strict local minimizer of $f(\cdot)$ in $\cbr{\pi^1: \norm{\pi^1 - \overline{\pi}^1} \leq \delta}$, 
$(c)$ follows from the definition of $\overline{\pi}^2$, and $(d)$ follows from the sub-optimality of $\overline{\pi}^1$ for $f(\cdot)$.
On the other hand, if $\pi^1 = \overline{\pi}^1 $, then we must have $\pi^2 \neq \overline{\pi}^2$, and hence 
\begin{align*}
V^{\pi^1, \pi^2} (s_1) 
\overset{(e)}{>}  f(\pi^1) 
=  f(\overline{\pi}^1) 
= V^{\overline{\pi}^1, \overline{\pi}^2} (s_1) 
>  \min_{\pi \revise{\in \mathrm{\Pi}_{\mathrm{MR}}}} V^\pi(s_1) ,
\end{align*}
where $(e)$ follows from that $\overline{\pi}^2$ is the unique minimizer of $\min_{\pi^2 \revise{\in \Delta_{\cA_2} }} V^{\overline{\pi}^1, \pi^2}(s_1)$. 
The above two observations jointly imply 
$
V^{\pi^1, \pi^2} (s_1)  > V^{\overline{\pi}^1, \overline{\pi}^2} (s_1)  >  \min_{\pi \revise{\in \mathrm{\Pi}_{\mathrm{MR}}}} V^\pi(s_1) .
$
This concludes the proof. 
\end{proof}

\begin{remark}\label{remark_stationarity_np_reduction}
It could be interesting to note that with a fixed transition kernel, the value function $V^\pi_P$ defined in \eqref{def_nonrobust_value} does not possess sub-optimal local minimizers when $\Pi = \Pi_{\mathrm{MR}}$ (cf. \cite[Lemma 4.1]{agarwal2021theory}). 
In particular, it is shown in \cite{agarwal2021theory} that a gradient-dominance condition holds for the non-robust value function $V^\pi_{P}$. 
Similar gradient-dominance conditions have been proposed in the context of two-player zero-sum Markov games \cite{daskalakis2020independent}, and later for   robust MDPs \eqref{def_static_rmdp} \cite{wang2023policy}, with the idea of replacing the non-smooth objective value by its Moreau envelope. 
Provided such a gradient dominance condition holds, one can establish the convergence of the subgradient method for these two classes of problems with an argument first presented in \cite{davis2019stochastic} for generic non-convex non-smooth optimization.
Unfortunately given Theorem \ref{thrm_strict_local_min}, it is clear that gradient-dominance would not hold for robust MDPs \eqref{def_static_rmdp} in the general setting.  

On the other hand, it could be worth noting that the approach of showing gradient dominance taken in \cite{daskalakis2020independent, wang2023policy} can be salvaged if the corresponding worst-case opponent (resp. nature)'s policy (resp. kernel) is unique. Although this does not hold directly for Markov games, one can replace the opponent's feasible set (probability simplex) by its strictly convex inner approximation bounded by $\epsilon$ in Hausdorff distance, which in turn guarantees the uniqueness of the opponent's worst-case policy via the dynamic programming equation.  
The gradient dominance  in \cite{daskalakis2020independent} can then be recovered by taking $\epsilon$ to $0$ and using the continuity of the proximal mapping. 
The same inner approximation approach could be used to establish the gradient dominance of robust MDPs \eqref{def_static_rmdp} with convex and compact $\mathrm{s}$-rectangular ambiguity sets. 
For \eqref{def_static_rmdp} with general ambiguity sets, we are not aware of a clean approach that can avoid assuming a unique worst-case kernel for every policy.
Note that such an assumption, on the other hand, would be quite challenging to verify due to the lack of dynamic programming equations. 
In particular, if the ambiguity set contains a finite number of kernels, then the uniqueness of the worst-case kernel for every policy in fact reduces the robust MDP to a non-robust MDP. 
\end{remark}

\begin{proposition}
Suppose $\cP = \cbr{P_{(1)}, \ldots, P_{(K)}}$. If for every policy the corresponding worst-case kernel is unique, then there exists $P^* \in \cP$ such that 
$V^\pi_{P^*}(s_1) > V^\pi_{P}(s_1)$ for any $P \in \cP \setminus \cbr{P^*}$ and any $\pi \in \Pi_{\mathrm{MR}}$. 
Consequently $V_r^\pi(s_1) \coloneqq \max_{1 \leq k \leq K} \{V^\pi_{P_{(k)}}(s_1) \} = V_{P^*}^\pi(s_1)$ for any $\pi \in \Pi_{\mathrm{MR}}$.
\end{proposition}

\begin{proof}
Let us define $\Pi_k = \{\pi \in \Pi_{\mathrm{MR}}: V^{\pi}_{P_{(k)}} (s_1) > \max_{k' \neq k} V^{\pi}_{P_{(k')}} (s_1) \}$.
Since $V^{\pi}_{P}(s_1)$ is continuous in $\pi$,  $\Pi_k$ is relatively open in $\Pi_{\mathrm{MR}}$. 
Suppose there exists $k^* \neq (k^*)'$ such that $\Pi_{k^*}$ and $\Pi_{(k^*)'}$ are both non-empty. 
Since the worst-case kernel is unique for every $\pi \in \Pi_{\mathrm{MR}}$, we have that $\cbr{\Pi_k}_{k=1}^K$ are disjoint,
and $\cup_{1 \leq k \leq K} \Pi_k = \Pi_{\mathrm{MR}}$.
Combining this with each $\Pi_k$ being relatively open, we have that  $\Pi_{k^*}$ and $\cup_{{k}' \neq k^*} \Pi_{{k}'}$ are relatively open and disjoint, and their union is $\Pi_{\mathrm{MR}}$. This contradicts with the fact that  $\Pi_{\mathrm{MR}}$ is~connected. 
\end{proof}

When the maximization problem in \eqref{static_rmdp_two_kernels} involves both ambiguous cost functions $\cbr{c_t}$ and transition kernel $P$, \cite{le2007robust} \revise{proposes an approach that attempts to establish the NP-completeness of \eqref{static_rmdp_two_kernels} based on the reduction from Path with Forbidden Pairs \cite{gabow1976two}.}
It appears to us that the reduction presented in \cite{le2007robust} is  not polynomial, as the transition kernel constructed therein needs to memorize the visitation history of the forbidden pairs and hence is non-Markovian.

\begin{remark}\label{remark_multi}
While our focus is on ambiguity sets with two transition kernels,  if \eqref{def_static_rmdp} involves $\cP$ consisting of $K$ distinct kernels and the complexity characterization is allowed to depend on $K$,  \cite{le2007robust} points out that finding the optimal randomized Markovian policy for \eqref{def_static_rmdp}  can be shown to be NP-complete by reducing from the SAT problem.  
A similar and explicit reduction from 3-CNF-SAT can be found in  \cite{steimle2021multi}, which is originally developed for showing the NP-completeness of multi-model MDPs. 
\revise{Note that the approach in \cite{steimle2021multi} also allows the cost function to be ambiguous along with the transition kernel.}
It might be  worth noting here that the NP-completeness for robust MDPs with $K$ kernels does not imply the existence of local minimizers or the hardness of \eqref{static_rmdp_two_kernels}. 
\end{remark}

By adopting the set partition problem again, one can indeed strengthen Theorem \ref{thrm_np_nonrandom} by establishing the NP-hardness of finding the optimal policy among Markovian randomized policies $\Pi_{\mathrm{MR}}$.

\begin{theorem}[NP-hardness for $\Pi_{\mathrm{MR}}$]
The robust MDP  \eqref{static_rmdp_two_kernels} is NP-hard
when $\Pi = \Pi_{\mathrm{MR}}$. 
\end{theorem}

\begin{proof}
Similar to Theorem \ref{thrm_np_nonrandom}, let us fix $\cW = \cbr{w_1, w_2, \ldots, w_n}\subset \ZZ_+$, and consider the following robust MDP instance. 
Define $\cS = \cup_{t=1}^n \cS_t$, where 
$\cS_t = \cbr{A_t, B_t, C_t, D_t}$.
At each $A_t$ and $B_t$, there are two actions $\cbr{L, R}$, respectively, and there is no available action at other states.
The cost function only depends on the state and is given by  
\begin{align*}
c_t(B_t) = w_t, c_t(C_t) = 1, ~ c_t(A_t) = c_t(D_t) = 0.
\end{align*}
Consider the transition kernels $P_{(1)}$ and $P_{(2)}$ given as follows, 
\begin{align*}
& P_{(1)}(B_t| A_t, L) = 1, ~ P_{(1)}(D_t | A_t, R) = 1, ~ P_{(1)}(C_t|B_t, L) =  1, P_{(1)}(D_t|B_t, R) = 1,  \\ &  P_{(1)}(D_t |C_t) = 1, P_{(1)}(A_{t+1} |D_t) = 1; \\
& P_{(2)}(B_t| A_t, R) = 1, ~ P_{(2)}(D_t | A_t, L) = 1, ~ P_{(2)}(C_t|B_t, R) =  1, P_{(2)}(D_t|B_t, L) = 1,  \\ &  P_{(2)}(D_t |C_t) = 1, P_{(2)}(A_{t+1} |D_t) = 1.
\end{align*} 
Note that $P_{(1)}$ and $P_{(2)}$ differ from each other only at $A_t$ and $B_t$.
For any $\pi \in \Pi_{\mathrm{MR}}$, define 
$\pi_{t, A} = \pi_t(L|A_t), \pi_{t, B} = \pi_t(L | B_t)$, then we have 
\begin{align}\label{np_mr_value}
V^{\pi}_{P_{(1)}} (A_1) = \tsum_{t=1}^n \pi_{t,A} (w_t + \pi_{t,B} ), ~
V^{\pi}_{P_{(2)}} (A_1)  = \tsum_{t=1}^n (1- \pi_{t,A}) \sbr{ w_t + (1 - \pi_{t, B}) }.
\end{align}
Hence we have 
$
V^{\pi}_{P_{(1)}} (A_1) + V^{\pi}_{P_{(2)}} (A_1)   = \tsum_{t=1}^n w_t + \tsum_{t=1}^n \sbr{\pi_{t,A} \pi_{t,B} + (1-\pi_{t,A}) (1-\pi_{t,B})} 
\geq 
\tsum_{t=1}^n w_t$.
This in turn implies that 
\begin{align*}
\min_{\pi \in \Pi_{\mathrm{MR}}}  V^\pi_r(A_1) & = \min_{\pi \in \Pi_{\mathrm{MR}}} \max \cbr{V^{\pi}_{P_{(1)}} (A_1), V^{\pi}_{P_{(2)}} (A_1)} 
\overset{(a)}{\geq} \frac{\tsum_{t=1}^n w_t}{2}.
\end{align*}
We claim that the equality holds in $(a)$ if and only if there exists $\cN \subset \cbr{1,\ldots, n}$ such that 
$\tsum_{t \in \cN}  w_t =  \tsum_{t \notin \cN}  w_t $. 
Indeed, suppose $\cN \subset \cbr{1,\ldots, n}$ satisfies  $\tsum_{t \in \cN} w_t = \tsum_{t \notin \cN}  w_t $ for some $\cN \subset \cbr{1, \ldots, n}$. 
Then one can consider policy $\pi^*$ with $\pi^*_{t,A} = 1$ and $\pi^*_{t,B} = 0$ for $t \in \cN$, together with $\pi^*_{t,A} = 0$ and $\pi^*_{t, B} = 1$ for $t \notin \cN$. 
In this case, we have from \eqref{np_mr_value} that  
$
V^{\pi^*}_{P_{(1)}}(A_1) = \tsum_{t \in \cN} w_t, ~ V^{\pi^*}_{P_{(2)}}(A_1) = \tsum_{t \notin \cN} w_t,
$
and consequently $V^{\pi^*}_r(A_1) =  \rbr{\tsum_{t=1}^n w_t}/{2}$. 
Clearly $\pi^*$ is an optimal policy and $(a)$ holds with equality.
On the other hand, suppose equality holds in $(a)$, then 
the corresponding optimal policy $\pi^* \in \Pi_{\mathrm{MR}}$ must be non-randomized and satisfies $(\pi^*_{t,A}, \pi^*_{t, B}) \in \cbr{(1,0), (0,1)}$ for $1 \leq t \leq n$.
In this case, define $\cN = \cbr{t: \pi^*_{t,A} = 1 }$,  then one can readily see from \eqref{np_mr_value} that 
$
V^{\pi^*}_{P_{(1)}}(A_1) = \tsum_{t \in \cN} w_t, V^{\pi^*}_{P_{(2)}}(A_1) = \tsum_{t \notin \cN} w_t.
$
Combining this with $V^{\pi^*}_r(A_1) = \rbr{\tsum_{t=1}^n w_t}/{2}$, we obtain 
$\tsum_{t \in \cN} w_t = \tsum_{t \notin \cN}  w_t $.
\end{proof}


\paragraph{Dynamic Formulation.}

As discussed in Section \ref{sec_intro}, in the static DRMDPs \eqref{static_rmdp_two_kernels}, the transition kernel $P \in \cP$ is selected before the realization of the process $\cbr{(S_t, A_t)}_{t=1}^T$, and the non-rectangularity of $\cP$ creates coupling between transition probabilities across states. 
This unfortunately destroys the dynamic decomposition of the corresponding distributionally robust functional in \eqref{static_rmdp_two_kernels}, and  the dynamic programming equations no longer exist. 
\revise{This is to our belief the root source of the hardness results presented in Section~\ref{sec_nonrandomized}.}

\revise{
To ensure the existence of dynamic programming equations, we can consider the dynamic formulation of DRMDPs \cite{li2025rectangularity} that allows an adversarial nature to adopt a policy that selects the transition kernel at every state, based on the process history leading up to the current state.
Unlike static DRMDPs that require rectangularity of $\cP$,  
for dynamic DRMDPs one can always obtain the dynamic equations without such requirements.
That is, the optimal value functions $\cbr{V_t}$, and the robust value functions of a given policy $\cbr{V^\pi_t}$ in the dynamic DRMDPs  are  characterized by  
\begin{align*}
V_{t}(s) & = \min_{\pi \in \Pi} \max_{P \in \cP} \tsum_{a_t \in \cA_t}  \pi_t(a_t |s_t) \sbr{
c_t(s_t, a_t) +  \tsum_{s_{t+1} \in \cS_{t+1}} P(S_{t+1} = s_{t+1} | S_t = s_t, A_t = a_t) V_{t+1}(s_{t+1}) 
}, \\
V^\pi_{t}(s) & = \max_{P \in \cP} \tsum_{a_t \in \cA_t}  \pi_t(a_t |s_t) \sbr{
c_t(s_t, a_t) +  \tsum_{s_{t+1} \in \cS_{t+1}} P(S_{t+1} = s_{t+1} | S_t = s_t, A_t = a_t) V^\pi_{t+1}(s_{t+1}) 
}
\end{align*}
for general compact $\cP$.
In contrast to the static DRMDPs,  computing the optimal value functions and policies is $R$-polynomial 
for both 
$\Pi = \Pi_{\mathrm{MR}}$ and $\Pi = \Pi_{\mathrm{MD}}$
when $\cP$ is convex, compact, and equipped with a separating oracle.
One can also consider establishing the gradient-dominance condition for the above value function $V^{\pi}_t$ when $\cP$ is convex and closed,  with the inner approximation argument discussed in Remark \ref{remark_stationarity_np_reduction}.
}


\section{Concluding Remarks}

We conclude by noting that the hard instances constructed within the manuscript can be extended to the discounted infinite-horizon setting by taking the union of the per-stage state spaces and introducing an additional sink state. 
That is,  $\cS = \cup_{1 \leq t \leq T} \cS_t \cup \cS_{T+1}$, 
where $\cbr{\cS_t}_{t=1}^T$ is defined as before and $\cS_{T+1} = \cbr{s_{\mathrm{sink}}}$ denotes an additional sink state. 
The transition kernel $P_{(i)}$ in the discounted infinite-horizon setting is defined in the same way as in the original instances, with $P_{(i)}(S_{T+1} = s_{\mathrm{sink}} | S_{T} ) =1$ for $i \in \cbr{1,2}$. 
It remains interesting to study whether similar results can be established for the infinite-horizon average-cost setting.

\section*{Acknowledgements}
We would like to thank Ehsan Sharifian for helpful discussions that inspire the development of this manuscript.
We are also grateful to Professor Alexander Shapiro for his valuable suggestions on the initial draft, and to Nian Si for bringing to our attention the related developments in \cite{le2007robust}.

\bibliographystyle{plain}
\bibliography{references}

@article{xu2010distributionally,
  title={Distributionally robust Markov decision processes},
  author={Xu, Huan and Mannor, Shie},
  journal={Advances in Neural Information Processing Systems},
  volume={23},
  year={2010}
}

@article{li2025rectangularity,
  title={Rectangularity and Duality of Distributionally Robust Markov Decision Processes: Y. Li, A. Shapiro},
  author={Li, Yan and Shapiro, Alexander},
  journal={Mathematical Programming},
  pages={1--42},
  year={2025},
  publisher={Springer}
}

@article{steimle2021multi,
  title={Multi-model Markov decision processes},
  author={Steimle, Lauren N and Kaufman, David L and Denton, Brian T},
  journal={IISE Transactions},
  volume={53},
  number={10},
  pages={1124--1139},
  year={2021},
  publisher={Taylor \& Francis}
}

@article{gabow1976two,
  title={On two problems in the generation of program test paths},
  author={Gabow, Harold N. and Maheshwari, Shachindra N and Osterweil, Leon J.},
  journal={IEEE Transactions on Software Engineering},
  number={3},
  pages={227--231},
  year={1976},
  publisher={IEEE}
}

@incollection{Tallec,
  author={Y. Le Tallec},
  title={Robust, {R}isk-{S}ensitive,
 and {D}ata-{D}riven {C}ontrol of {M}arkov {D}ecision
{P}rocesses},
  publisher={Ph.D. thesis, Massachusetts Institute of Technology. Cambridge, MA},
  year={2007},
}

@article{daskalakis2020independent,
  title={Independent policy gradient methods for competitive reinforcement learning},
  author={Daskalakis, Constantinos and Foster, Dylan J and Golowich, Noah},
  journal={Advances in neural information processing systems},
  volume={33},
  pages={5527--5540},
  year={2020}
}

@article{davis2019stochastic,
  title={Stochastic model-based minimization of weakly convex functions},
  author={Davis, Damek and Drusvyatskiy, Dmitriy},
  journal={SIAM Journal on Optimization},
  volume={29},
  number={1},
  pages={207--239},
  year={2019},
  publisher={SIAM}
}

@inproceedings{wang2023policy,
  title={Policy gradient in robust mdps with global convergence guarantee},
  author={Wang, Qiuhao and Ho, Chin Pang and Petrik, Marek},
  booktitle={International conference on machine learning},
  pages={35763--35797},
  year={2023},
  organization={PMLR}
}

@article{agarwal2021theory,
  title={On the theory of policy gradient methods: Optimality, approximation, and distribution shift},
  author={Agarwal, Alekh and Kakade, Sham M and Lee, Jason D and Mahajan, Gaurav},
  journal={Journal of Machine Learning Research},
  volume={22},
  number={98},
  pages={1--76},
  year={2021}
}

@article{sutton1999policy,
  title={Policy gradient methods for reinforcement learning with function approximation},
  author={Sutton, Richard S and McAllester, David and Singh, Satinder and Mansour, Yishay},
  journal={Advances in neural information processing systems},
  volume={12},
  year={1999}
}

@article{Nilim2005,
  title={Robust control of {M}arkov decision processes
   with uncertain transition probabilities},
  author={A. Nilim and L. El Ghaoui},
  journal={Operations Research},
  volume={53},
  pages={780--798},
  year={2005}
}

@Article {iyen,
  Title                    = {Robust {D}ynamic {P}rogramming},
  Author                   = {G.N. Iyengar},
  Journal                  = {Mathematics of Operations Research},
  Year                     = {2005},
  Pages                    = {257--280},
  Volume                   = {30}
}

@article{tulcea1949mesures,
  title={Mesures dans les espaces produits},
  author={Tulcea, CT Ionescu},
  journal={Atti Accad. Naz. Lincei Rend},
  volume={7},
  pages={208--211},
  year={1949}
}

@phdthesis{le2007robust,
  title={Robust, risk-sensitive, and data-driven control of Markov decision processes},
  author={Le Tallec, Yann},
  year={2007},
  school={Massachusetts Institute of Technology}
}

@article{wiesemann2013robust,
  title={Robust Markov decision processes},
  author={Wiesemann, Wolfram and Kuhn, Daniel and Rustem, Ber{\c{c}}},
  journal={Mathematics of Operations Research},
  volume={38},
  number={1},
  pages={153--183},
  year={2013},
  publisher={INFORMS}
}

\end{document}